\documentclass[conference]{IEEEtran}
\usepackage{command_MU}
\usepackage{latexsym}
\usepackage{amsmath,amsfonts, amssymb}
\usepackage{hyperref}
\usepackage{booktabs}
\usepackage{colortbl}
\usepackage{algorithm}
\usepackage{algpseudocode}
\usepackage{lscape}
\usepackage{mathtools}
\usepackage{array}
\usepackage[caption=false,font=normalsize,labelfont=sf,textfont=sf]{subfig}
\usepackage{textcomp}
\usepackage{stfloats}
\usepackage{url}
\usepackage{verbatim}
\usepackage{graphicx}
\usepackage{cite}
\usepackage{tikz}
\renewcommand\thesubsection{\thesection.\arabic{subsection}}
\IEEEoverridecommandlockouts

\DeclareMathOperator*{\argmin}{arg\,min}

\newcommand{\euler}{\mathrm{e}}
\newcommand{\imgn}{\mathrm{j}}
\def\BibTeX{{\rm B\kern-.05em{\sc i\kern-.025em b}\kern-.08em
    T\kern-.1667em\lower.7ex\hbox{E}\kern-.125emX}}
\begin{document}

\title{Comparison of 2D Regular Lattices \\for the CPWL Approximation of Functions
\thanks{This work was supported by the European Research Council (ERC Project FunLearn) under Grant 101020573, in part by the Swiss National Science Foundation, Grant $200020\_219356$.}
}

\author{\IEEEauthorblockN{1\textsuperscript{st} Mehrsa Pourya}
\IEEEauthorblockA{
\textit{EPFL}\\
Lausanne, Switzerland \\
mehrsa.pourya@epfl.ch}
\and
\IEEEauthorblockN{2\textsuperscript{nd} Maïka Nogarotto}
\IEEEauthorblockA{
\textit{EPFL}\\
Lausanne, Switzerland \\
maika.nogarotto@epfl.ch}
\and
\IEEEauthorblockN{3\textsuperscript{rd} Michael Unser}
\IEEEauthorblockA{
\textit{EPFL}\\
Lausanne, Switzerland \\
michael.unser@epfl.ch}
}

\maketitle
\begin{abstract}
We investigate the approximation error of functions with continuous and piecewise-linear (CPWL) representations. We focus on the CPWL search spaces generated by translates of box splines on two-dimensional regular lattices. We compute the approximation error in terms of the stepsize and angles that define the lattice. Our results show that hexagonal lattices are optimal, in the sense that they minimize the asymptotic approximation error.
\end{abstract}
\section{Introduction \label{sec:intro}}
\IEEEPARstart{C}{ontinuous} and piecewise-linear (CPWL) representations play a fundamental role in signal processing, computer graphics, and computational mathematics \cite{Eriksson2004, ciarlet1978finite, pourya2023delaunay}. They are widely appreciated for their simplicity, computational efficiency, and ability to approximate complex structures. They are also of interest in machine learning because they encompass the same class of functions generated by neural networks with the rectified linear unit (ReLU) activation functions \cite{arora2016understanding}. Box splines extend univariate B-splines to multiple dimensions and provide a structured framework for the construction of CPWL search spaces \cite{de2013box, KIM2024128376, goujon2022stable, pourya2024continuous, pourya2024box}. 

A significant body of research investigates the approximation error associated with CPWL representations \cite{bertoluzza2012primer, devore1998nonlinear}. However, much of this work focuses on the upper bounds of the error, and less attention is paid to the exact form of the asymptotic error. Fourier-domain methods such as \cite{blu1999quantitative} provide powerful tools to analyze such asymptotic behaviors. However, their final results are limited to the one-dimensional scenario. 

In this paper, we focus on the two-dimensional case. We rely on box splines to construct each CPWL function over a domain that is partitioned by triangulations where the vertices are on regular lattices. These lattices and the edges of the triangulation form an underlying grid for each CPWL function. Cartesian and hexagonal grids are the well-known examples. Cartesian grids are simple to use, whilst hexagonal grids are known to have better sampling properties \cite{van2004hex, campos2021learning}. Our goal is to investigate the effect of the grid on the approximation error. Our main contributions are as follows.
\begin{enumerate}
    \item Computation of the approximation error for a general grid: We compute the approximation error for a box-spline-based CPWL search space with a general grid in terms of its angles and stepsize. We provide this result in Theorem \ref{the:error_computation}. In Theorem \ref{the:error_upper_bound}, we present an upper bound for the dominant term of this error in the asymptotic regime. Notably, we compute an asymptotic error constant that depends on the angles that define the grid. 
    \item Optimality of the hexagonal grid: We show that the asymptotic error constant is minimized for the hexagonal grid. We present examples that validate this result. 
    \item Relation to ReLU neural networks: We provide in Theorem \ref{the:box_spline_as_relu} a concise representation of box splines as two-layer neural networks with ReLU activation functions. To our knowledge, this is the simplest construction of two-dimensional box splines with ReLU functions.  
\end{enumerate}

In Section \ref{sec:pre}, we define box splines and their associated CPWL search spaces, and present our proposed parameterization for the grid. We then formally define the approximation error. In Section \ref{sec:methods}, we present our results for the computation and analysis of this approximation error.


\section{Preliminaries \label{sec:pre} and Problem Formulation}
\subsection{Continuous and Piecewise-Linear (CPWL) Box Splines \label{part:def_cart_hex}}
The two-dimensional CPWL box spline \smash{$B_{\M \Xi} \colon \R^2 \to \R$} is defined through three vectors $\{\V \xi_n\}_{n=1}^{3}$, where $\{\V \xi_n\}_{n=1}^{2}$ are linearly independent and $\V \xi_{3} = \V \xi_1 +  \V \xi_2$. We refer to $\M \Xi$ as the grid matrix and define it as $\M \Xi = [\V \xi_1 \  \V \xi_2] \in \R^{2 \times 2}$. The box spline in the Fourier domain is expressed as 
\begin{equation}
    \label{eq:def_box_spline}
    \hat{B}_{\M \Xi}(\V{\omega}) = \abs{\det \M \Xi}\prod_{r=1}^{3} \mathrm{sinc}\left(\frac{{ \V \xi_r^{\top} \V{\omega} }}{2}\right),
\end{equation}
where $\abs{\det \M \Xi}$ is the determinant of $\M \Xi$ and $\mathrm{sinc}(\omega) = \frac{\sin \omega}{\omega}$. 

If we choose $\M \Xi = \M I$ (identity matrix), we obtain the Cartesian box spline with Fourier transform
\begin{equation}
    \label{eq:cartesian_fourier}
    \hat{\varphi}(\V{\omega}) =  \mathrm{sinc}\left(\frac{\omega_1 + \omega_2}{2}\right) \prod_{r=1}^{2} \mathrm{sinc}\left(\frac{\omega_r}{2}\right).
\end{equation}
The formula for the spatial evaluation of $\varphi$ ensues, with 
\begin{equation}
    \label{eq:cartesian_spatial}
    \varphi(\V x) = \max\big(1 + \min(x_1, x_2, 0) -  \max(x_1, x_2, 0), 0 \big),
\end{equation}
taken from \cite{pourya2024box}. Any CPWL box spline $B_{\M \Xi}$ can be expressed in terms of the Cartesian box spline $\varphi$ as $B_{\M \Xi}(\V x) = \varphi(\M \Xi^{-1} \V x)$. This follows from a simple change of variable in the Fourier domain. In Figure \ref{fig:boxsplines}, we present two examples of box splines: Cartesian box spline with $ \M \Xi_{\mathrm{Cart}} = T \M I$, and hexagonal box spline with $ \M \Xi_{\mathrm{Hex}} = T \M D_{\mathrm{Hex}}$ with $\M D_{\mathrm{Hex}} = ({\frac{\sqrt{3}}{2}})^{-0.5} \begin{bmatrix}
    1 & -0.5\\
    0 & \frac{\sqrt{3}}{2}
\end{bmatrix}$ for some $T > 0$.

\subsection{CPWL Search Space}
We define the CPWL search space
\begin{equation}
    \mathcal{V}_{\M \Xi} = \left\{\sum_{\V k \in \Z^2} c\left[\V k\right] B_{\M \Xi}\left(\V \cdot - \M \Xi \V k\right) \colon c\left[\V \cdot\right] \in \ell_2(\Z^2) \right\}, 
\end{equation}
or equivalently ,
\begin{equation}
    \label{eq:search_space_with_phi}
    \mathcal{V}_{\M \Xi} = \left\{ \sum_{\V k \in \Z^2} c\left[\V k\right] \varphi\left(\M \Xi ^{-1} \V \cdot - \V k\right) \colon c\left[\V \cdot\right] \in \ell_2(\Z^2) \right\}.
\end{equation}
  There, we have used the Cartesian box spline $\varphi$ and the relation $B_{\M \Xi}(\V x) = \varphi(\M \Xi^{-1} \V x)$. The translated basis functions $\{B_{\M \Xi}(\V \cdot - \M \Xi \V k)\}_{\V k \in \Z^2}$ form a Riesz basis of $\mathcal{V}_{\M \Xi}$, which guarantees a stable link between each mapping $s \colon \R^2 \to \R$ and its expansion coefficients $c$. It can further reproduce any affine mapping \cite{goujon2022stable}. The domain of each function $s \in \mathcal{V}_{\M \Xi}$ thus consists of triangles with vertices located on the regular lattice $\{\M \Xi \V k\}_{\V k \in \Z^2}$. This lattice and the edges of the triangulation form an underlying grid for each function. The grid lines are parallel to the directions $\V \xi_1, \V \xi_2$, and $\V \xi_3 = \V \xi_1 + \V \xi_2$ where we had that $\M \Xi = [\V \xi_1 \  \V \xi_2]$. 

\subsection{Parameterization of the Grid Matrix}
We parameterize $\M \Xi$ by the two angles $\theta_1, \theta_2$, and the stepsize $T$ as
\begin{equation}
    \label{eq:grid_def}
    \M \Xi = \frac{T}{\sqrt{\sin(\theta_2 - \theta_1)}} \begin{bmatrix}
        \cos{\theta_1} & \cos{\theta_2} \\
        \sin{\theta_1} & \sin{\theta_2} 
    \end{bmatrix}. 
\end{equation}
 We define $\delta = \left(\theta_2 - \theta_1\right)$, and assume that $0 < \theta_1, \delta < 2\pi$ (without loss of generality). It follows that $\abs{\det \M \Xi} = T^2$, which ensures that the number of grid points per area remains invariant for grids constructed with different angles. This invariance allows for a fair comparison of grids. Our parameterization encompasses the Cartesian grid matrix $\M \Xi_{\mathrm{Cart}}$ with $\theta_1 = 0, \theta_2=\frac{\pi}{2}$, and the hexagonal grid matrix $\M \Xi_{\mathrm{Hex}}$ with $\theta_1 = 0, \theta_2=\frac{2\pi}{3}$. We represent the central parts of the grids constructed with $\M \Xi_{\mathrm{Cart}}$ and $\M \Xi_{\mathrm{Hex}}$ in Figure \ref{fig:grids}. 
\subsection{Formulation of the Problem }
For $f \in L_2(\R^2)$, we are interested in the minimum $L_2$-error solution 
\begin{equation}
    f_{\text{CPWL}} \coloneqq \argmin_{s \in \mathcal{V}_{\M \Xi}} \norm{f - s}_{L_2}.
\end{equation}
\begin{figure}[hpt]
\begin{center}
  \includegraphics[width=7cm,height=5cm,keepaspectratio]{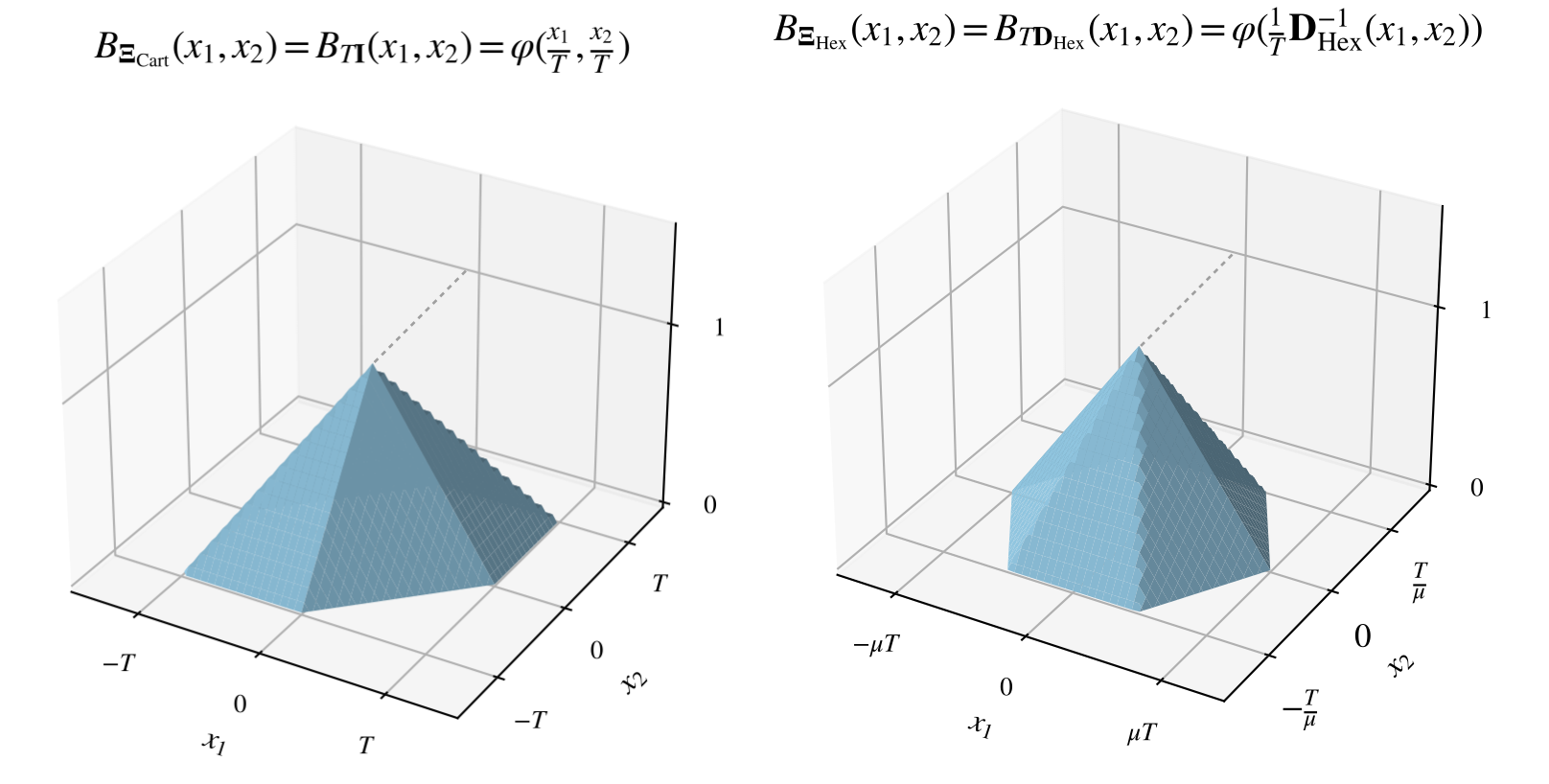}
  \caption{Cartesian (left) and hexagonal ($\mu = ({\frac{\sqrt{3}}{2}})^{-0.5}$) (right) box splines.}
  \label{fig:boxsplines}
\end{center}
\end{figure}

\begin{figure}[hpt]
\begin{center}
  \includegraphics[width=7cm,height=5cm,keepaspectratio]{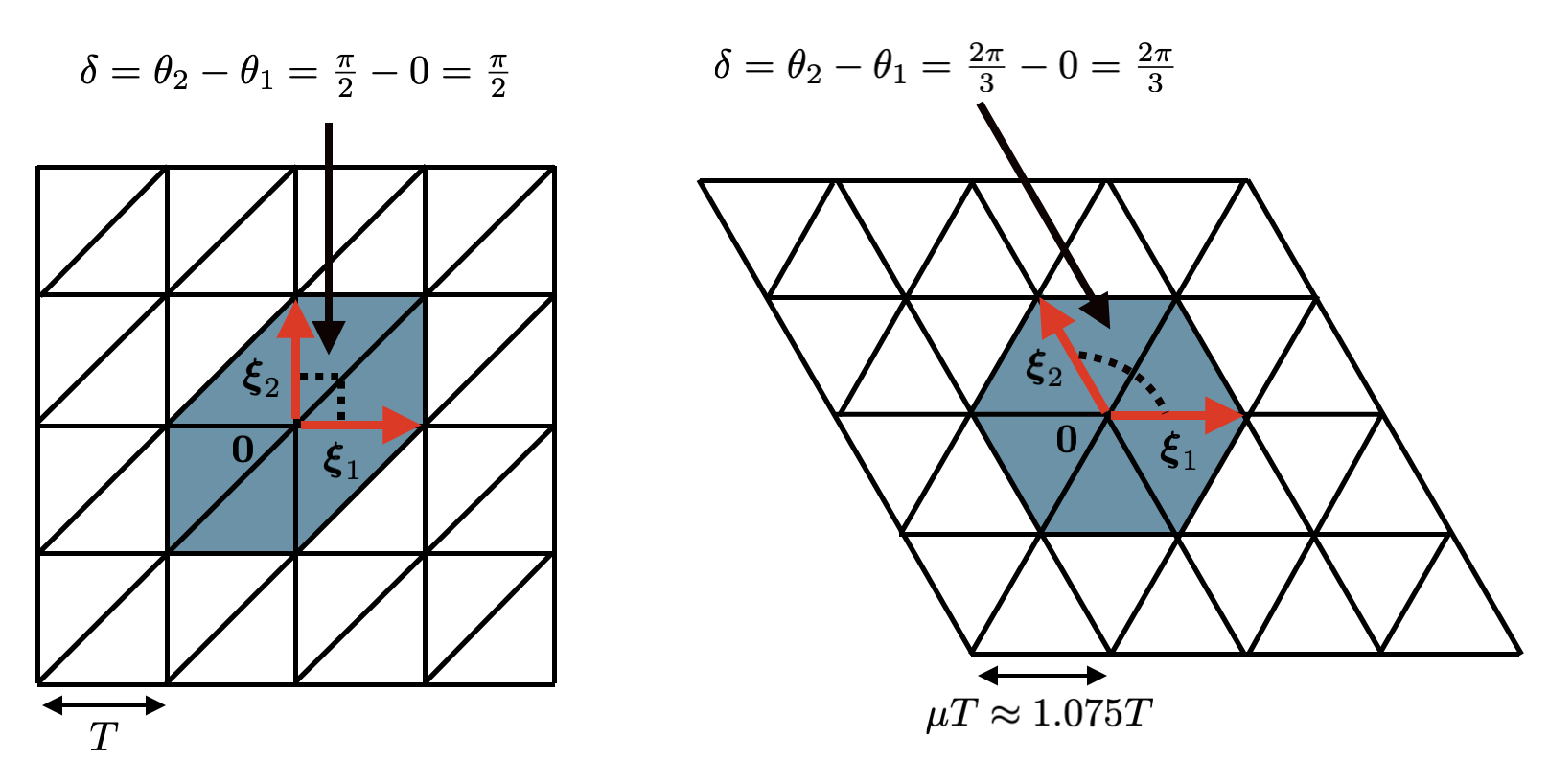}
  \caption{Cartesian grid with $\M \Xi_{\mathrm{Cart}} = [\V \xi_1 \  \V \xi_2] = \begin{bmatrix}
      T & 0 \\
      0 & T
  \end{bmatrix}$ (left) and hexagonal grid with $\M \Xi_{\mathrm{Hex}} = [\V \xi_1 \  \V \xi_2] = \begin{bmatrix}
      \mu T   & -0.5\mu T \\
      0 & 0.5\sqrt{3} \mu T
  \end{bmatrix}$ and $\mu = ({\frac{\sqrt{3}}{2}})^{-0.5}$ (right). The highlighted region depicts the support of the corresponding box splines $B_{\Xi_{\mathrm{Cart}}}$ and $B_{\Xi_{\mathrm{Hex}}}$.}
  \label{fig:grids}
\end{center}
\end{figure}
\vspace{-3mm}
This solution could be computed through the projector $P_{\mathcal{V}_{\M \Xi}}$ as 
\begin{align}
    &f_{\text{CPWL}} (\V x) =  P_{\mathcal{V}_{\M \Xi}} \{f\} (\V x) \\
   &= \sum_{\V k \in \Z^2} \left\langle f, \frac{1}{\abs{\det \M \Xi}} \phi (\M \Xi^{-1} \V\cdot - \V k)\right\rangle_{L_2} \varphi(\M \Xi^{-1} \V x - \V k),
\end{align}
which is taken from \cite{ramani2009} and is a classic result in approximation theory \cite{DEBOOR199437, 489025, BLU1999219}. The dual basis $\phi$ is defined through its Fourier transform
\begin{equation}
    \phi (\V \omega) = \frac{\hat{\varphi}(\V \omega)}{A_{\varphi} (\V \omega)},
\end{equation}
with
\begin{equation}
    \label{eq:def_a_phi}
    A_{\varphi}(\V \omega) = \sum_{\V k \in \Z^2} \abs{\hat{\varphi}(\V \omega + 2 \pi \V k)}^2.
\end{equation}
We define 
\begin{equation}
    \label{eq:define_e_f}
    \epsilon_{\M \Xi}(f)= \norm{f - P_{\mathcal{V}_{\M \Xi}} \{f\}}_{L_2}.
\end{equation}
In this paper, we want to quantify the error $e_{\M \Xi}(f)$ in terms of the stepsize $T$ and the angles $\theta_1$ and $\theta_2$ found in \eqref{eq:grid_def}. In simple words, we want to quantify the effect of the grid on the approximation error.  

\section{Methods \label{sec:methods}}
We first present our results for the computation of the error $ \epsilon_{\M \Xi}(f)$ and define an asymptotic error constant that depends on the angles of the grid. Then, we investigate the effect of the grid angles on the approximation error. Mainly, we show that the asymptotic error constant is minimized for hexagonal grids. Finally, we relate our findings to the approximation that results from ReLU neural networks. 
\subsection{Computation of the Approximation Error \label{part:a}}
 We first present an explicit formula for $A_\varphi$ in Proposition \ref{prop:A_phi}, which is crucial for the calculation of the error. Then, in Theorem \ref{the:error_computation}, we provide our result for the computation of $\epsilon_{\M \Xi}(f)$. Finally, in Theorem \ref{the:error_upper_bound}, we provide an upper bound for the dominant term which involves the asymptotic error constant $C(\theta_1, \theta_2)$.
\begin{proposition}
    \label{prop:A_phi}
    Let $\varphi$ be the Cartesian box spline. Then, one has that
    \begin{equation}
        \label{eq:a_phi_cartesian}
        A_{\varphi} (\V \omega) = \frac{1}{2} + \frac{1}{6} \big(\cos (\omega_1) + \cos(\omega_2) + \cos(\omega_1 + \omega_2)\big).
    \end{equation}
\end{proposition}
\begin{proof}
    The cumbersome computation of the infinite sum in \eqref{eq:def_a_phi} is simplified by the equality
\begin{equation}
   \sum_{\V k \in \Z^2} \abs{\hat{\varphi}(\V \omega + 2 \pi \V k)^2} = \mathrm{DTFT}\Big\{\big(\varphi * \varphi(- \V \cdot)\big) \big(\V k\big)\Big\}.
\end{equation}
To compute the autocorrelation sequence $a[\V k] = \big(\varphi * \varphi(- \V \cdot)\big) \big(\V k\big), \ \V k \in \Z^2$, we use the spatial evaluation of $\varphi$ in \eqref{eq:cartesian_spatial} and compute the result integrals to obtain that
\begin{equation}
    a[\V k] = \begin{cases}
        \frac{1}{2}, & \V k = \V 0 \\
        \frac{1}{12}, & \V k \in \{0, 1\}^2 \setminus \{(0, 0)\}\\
        0, & \text{otherwise}.
    \end{cases}
\end{equation}
Now, to complete the proof and obtain \eqref{eq:a_phi_cartesian}, we use the definition $\mathrm{DTFT} \big\{a[\V k]\big\}(\omega) = \sum_{\V k \in Z^2} a[\V k] \euler ^{-\imgn \V k ^{\top} \V \omega}$ and the equality $\cos(\cdot) = \frac{1}{2}(\euler^{\imgn \cdot} + \euler^{- \imgn \cdot})$.
\end{proof}

\begin{theorem}
    For a band-limited $f \in W^{\rho}_2$ (Sobolov space of order $\rho$) with $\rho > 2$ and for a general grid $\M \Xi$ defined with $0< T < 1$ and $\theta_1$ and $\theta_2$ from \eqref{eq:grid_def}, it holds that
    \label{the:error_computation}
    \begin{equation} 
         \epsilon_{\M \Xi}(f) = \epsilon_{\M \Xi, \text{asym}}(f) + \mathcal{O}(T^{\min(\rho, 3)}),
    \end{equation}
    where
    \begin{equation}
        \label{eq:error_dominant}
        \epsilon_{\M \Xi, \text{asym}}(f) = \frac{T^2}{12\sqrt{5}} \big({\int_{\R^2} \langle \M H_f(\V x), \begin{bmatrix}
        \alpha & \gamma \\
        0 & \beta
    \end{bmatrix} \rangle^2 \mathrm{d} \V x\big)}^{\frac{1}{2}}.
    \end{equation}
There, $\M H_f$ denotes the Hessian of $f$ and 
\begin{align}
\label{eq:abg}
    \alpha &= \frac{\cos^2 (\theta_1) + \cos (\theta_1) \cos (\theta_2) + \cos^2 (\theta_2)}{\abs{\sin(\theta_2 - \theta_1)}}, \notag \\
     \beta &= \frac{\sin^2 (\theta_1) + \sin (\theta_1) \sin (\theta_2) + \sin^2 (\theta_2)}{\abs{\sin(\theta_2 - \theta_1)}}\notag, \\
   \gamma &= \frac{\sin (2\theta_1) + \sin (\theta_1 + \theta_2) + \sin (2\theta_2)}{\abs{\sin(\theta_2 - \theta_1)}}.
\end{align}

\end{theorem}
\begin{proof}
For a general grid $\M \Xi$, a change of variables lead to
\begin{equation}
    \label{eq:grid_change_error}
     \epsilon_{\M \Xi}(f)  =   \epsilon_{T\M I}\left(f\left(\frac{1}{T} \M \Xi \V \cdot\right)\right).
\end{equation}
From Section 2.4 of \cite{ramani2009}, with $f \in W^{\rho}_2$ and $\rho > 2$, it holds that
\begin{equation}
    \label{eq:err_errk}
    \epsilon_{T\M I}(f) = \epsilon_{\text{dom}, T \M I} (f) + O(T^\rho),
\end{equation}
where 
\begin{equation}
    \epsilon_{\text{dom}, T \M I}(f) = \Big[\frac{1}{4\pi^2} \int_{\R^2}  \epsilon_{\phi, \varphi} (\V \omega T) \abs{\hat{f}(\V \omega)}^2 \mathrm{d} \V \omega \Big]^{\frac{1}{2}}.
\end{equation}
The error kernel $\epsilon_{\phi, \varphi}$ is defined as
\begin{align}
    \label{eq:def_error_kernel}
    \epsilon_{\phi, \varphi} (\V \omega) = 1 - \frac{\abs{\hat{\varphi}(\V \omega)}^2}{A_{\varphi}(\V \omega)}.
\end{align}
In our case, from \eqref{eq:cartesian_fourier} and Proposition \ref{prop:A_phi}, \eqref{eq:def_error_kernel} simplifies to 
\begin{align}
    \epsilon_{\phi, \varphi} (\V \omega) &= 1 \\ & - \frac{\prod_{r=1}^2 \mathrm{sinc}^2(\frac{\omega_r}{2})  \mathrm{sinc}^2(\omega_1 + \omega_2)}{ \frac{1}{2} + \frac{1}{6} \big(\cos (\omega_1) + \cos(\omega_2) + \cos(\omega_1 + \omega_2)\big)}.
\end{align}
By a Taylor series around $\V 0$, we get that
\begin{equation}
    \epsilon_{\phi, \varphi} (\V \omega T) = \frac{T^4}{720} (\omega_1^2 + \omega_1 \omega_2 + \omega_2^2)^2 + \mathcal{O}(T^6).
\end{equation}
Then, it follows that
\begin{align}
    \epsilon_{\text{dom}, T \M I}(f)^2 =  \epsilon_{\text{Taylor}, T \M I}(f)^2 + \mathcal{O}(T^6),
\end{align}
where we used that $\int_{\R^2} \mathcal{O}(T^6)  \abs{\hat{f}(\V \omega)}^2 \mathrm{d} \V \omega = \mathcal{O}(T^6)$ due to $f$ being band-limited, and
\begin{equation}
     \epsilon_{\text{Taylor}, T \M I}(f)^2 = \smash{\frac{T^4}{2880\pi^2} \int_{\R^2} {\big((\omega_1^2 + \omega_1 \omega_2 + \omega_2^2) \hat{f}(\V \omega)\big)}^2 \mathrm{d} \V \omega}.
\end{equation}
From \eqref{eq:grid_change_error} and \eqref{eq:err_errk}, we have that
\begin{align}
    \epsilon_{\M \Xi}(f) 
    &= \epsilon_{\text{Taylor}, T \M I}(f(\frac{1}{T} \M \Xi \V \cdot)) + \mathcal{O}(T^{\rho}) + \mathcal{O}(T^{3}) \notag \\
    &= \epsilon_{\text{Taylor}, T \M I}(f(\frac{1}{T} \M \Xi \V \cdot)) + \mathcal{O}(T^{\min(\rho, 3)}).
\end{align}
Next, we define $\M D = \frac{1}{T} \M \Xi$ and compute
\begin{align}
    \label{eq:error_comp_fourier}
    &\epsilon_{\text{Taylor}, T \M I}(f(\frac{1}{T} \M \Xi \V \cdot))^2 = 
   \frac{1}{4\pi^2} \int_{\R^2} \epsilon_{\phi, \varphi} (\V \omega T)  {\big(\hat{f}(\M D^{-\top} \V \omega)\big)}^2 \mathrm{d} \V \omega \notag \\
   &=  \frac{1}{4\pi^2}\int_{\R^2} \epsilon_{\phi, \varphi} (\M D^{\top} \V z T)  \big({\hat{f}(\V z)\big)}^2 \mathrm{d} \V z \notag \\ &= \frac{T^4}{2880\pi^2} \int_{\R^2}\big({(\alpha z_1^2 + \gamma z_1 z_2 + \beta z_2^2) \hat{f}(\V z)\big)}^2\mathrm{d} \V z  \notag   \\ 
   &= \frac{T^4}{720} \int_{\R^2} \big(\alpha \frac{\partial^2 f(\V x)}{\partial x_1^2} + \gamma \frac{\partial^2 f(\V x)}{\partial x_1 \partial x_2} + \beta \frac{\partial^2 f(\V x)}{\partial x_2^2}\big)^2 \mathrm{d} \V x   \notag \\
   &=  \frac{T^4}{720} \int_{\R^2} \langle \M H_f(\V x), \begin{bmatrix}
        \alpha & \gamma \\
        0 & \beta
    \end{bmatrix} \rangle^2 \mathrm{d} \V x. 
\end{align}
\end{proof}
For $T<1$, $ \epsilon_{\M \Xi, \text{asym}}(f)$ is the dominant term of the error. Moreover, in the asymptotic case $T \to 0$, we have that $\epsilon_{\M \Xi}(f)=  \epsilon_{\M \Xi, \text{asym}}(f)$. Now, we present an upper bound for $\epsilon_{\M \Xi, \text{asym}}(f)$.

\begin{figure}[hpt]
\begin{center}
  \includegraphics[width=8cm,height=5cm,keepaspectratio]{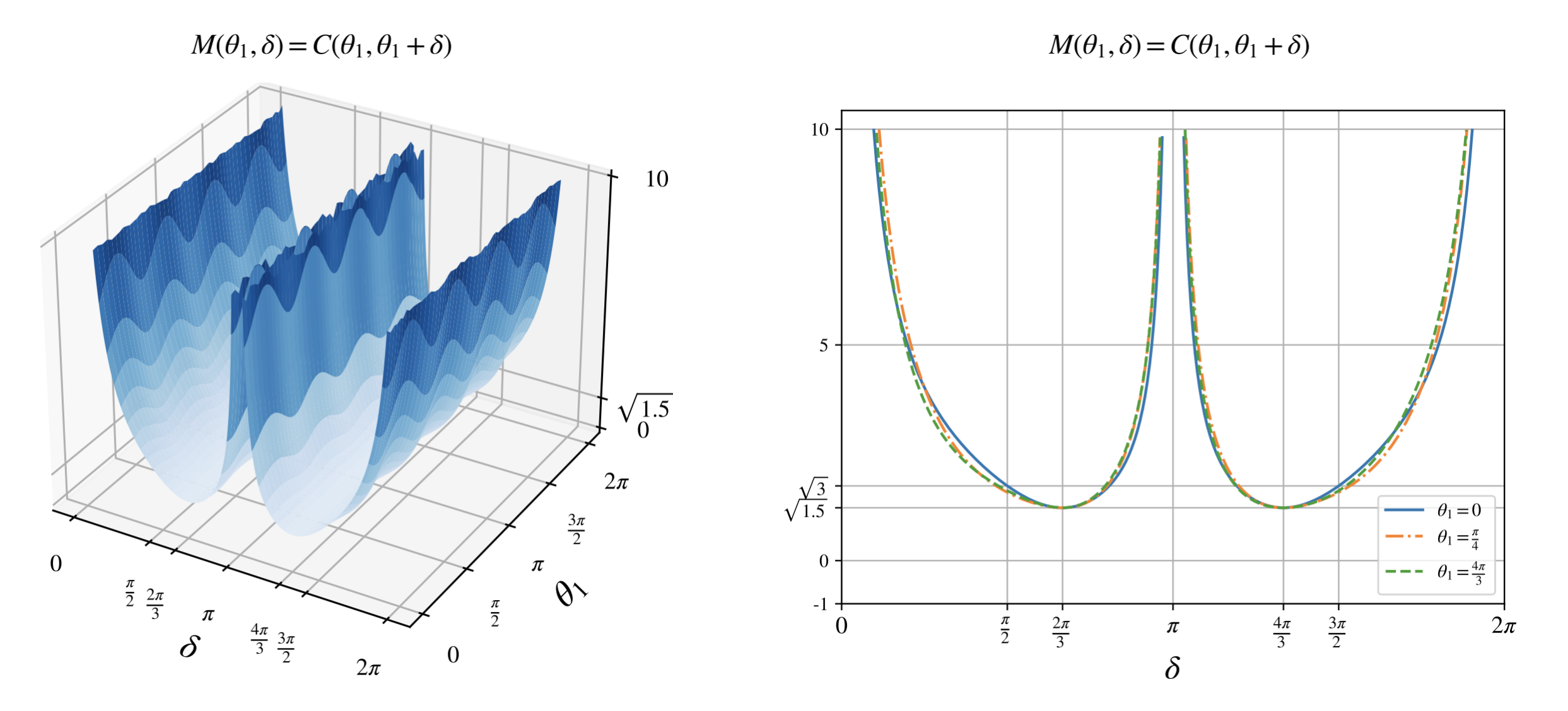}
  \caption{Error constant $M(\theta_1, \delta) = C(\theta_1, \theta_1 + \delta)$ in terms of $\theta_1$ and $\delta = \left(\theta_2 - \theta_1\right)$ (left) and one-dimensional profiles of $M(\theta_1, \delta)$ in terms of $\delta$ for $\theta_1 \in \{0, \frac{\pi}{4}, \frac{4\pi}{3}\}$ (right). In both plots, we only show $M(\theta_1, \delta)$ where $M(\theta_1, \delta) < 10$ for better visual representation. }
  \label{fig:error_plot}
\end{center}
\end{figure}
\begin{theorem}
    \label{the:error_upper_bound}
    For $\epsilon_{\M \Xi, \text{asym}}(f)$ under the same conditions as there in Theorem \eqref{the:error_computation}, it holds that
    \begin{equation}
     \epsilon_{\M \Xi, \text{asym}}(f) \leq  \frac{1}{12\sqrt{5}} C(\theta_1, \theta_2) T^2 \norm{\M H_f}_{F, L_2}, 
\end{equation}
where we define the asymptotic error constant $C(\theta_1, \theta_2) = {(\alpha ^2 + \beta^2 + \gamma^2)}^{\frac{1}{2}}$ through $\alpha, \beta$, and $\gamma$ given in \eqref{eq:abg}, along with the mixed norm $\norm{\M A(\V \cdot) }_{F, L_2} \coloneqq \big(\int_{\R^2} \norm{\M A(\V x)}_F^2\mathrm{d} \V x\big)^{\frac{1}{2}}$.
\end{theorem}

\begin{proof}
Hölder's inequality for matrices yields that
\begin{equation}
    \langle \M H_f(\V x), \begin{bmatrix}
        \alpha & \gamma \\
        0 & \beta
    \end{bmatrix} \rangle^2 \leq \norm{\M H_f(\V x)}_F^2 (\alpha ^2 + \beta^2 + \gamma^2).
\end{equation}
Then, through \eqref{eq:error_dominant}, we have that
\begin{align}
     \epsilon_{\M \Xi, \text{asym}}(f)^2 \leq \frac{T^4 (\alpha ^2 + \beta^2 + \gamma^2)}{720} \int_{\R^2} \norm{\M H_f(\V x)}_F^2\mathrm{d} \V x.
\end{align}
\end{proof} 
\subsection{Analysis of the Effect of the Grid \label{part:b}}
Now, we investigate the effect of the grid angles $\theta_1$ and $\theta_2$ on the error constant $C(\theta_1, \theta_2)$. In Figure \ref{fig:error_plot}, we plot $M(\theta_1, \delta) = C(\theta_1, \delta + \theta_1)$. The error constant takes higher values as $\delta$ approaches $0, \pi$, and $2\pi$.
More importantly, we observe that, for different values of $\theta_1$, $M(\theta_1, \delta)$ is minimized at $\delta = \frac{2 \pi}{3}$ or $\delta = \frac{4 \pi}{3}$. The minimal value of $C(\theta_1, \theta_2)$ is  $\sqrt{1.5}$. It is achieved when $\left(\theta_2 - \theta_1\right) = \frac{2\pi}{3}$ or $\left(\theta_2 - \theta_1 \right)= \frac{4\pi}{3}$, which corresponds to a hexagonal grid.  

We now present explicit formulas for $\epsilon_{\M \Xi, \text{asym}}(f)$ in the case of the Cartesian ($\M \Xi = \M \Xi_{\mathrm{Cart}}$) and hexagonal ($\M \Xi = \M \Xi_{\mathrm{Hex}}$) grids that are defined in Section \ref{part:def_cart_hex}. In these two cases, we illustrate the computation of $\epsilon_{\M \Xi, \text{asym}}(f)$ for a function $f$ whose Fourier response is a disk function. 

\subsubsection{Error for the Cartesian and Hexagonal Grids} 
For the Cartesian grid, we have that $\theta_1 = 0, \theta_2 = \frac{\pi}{2}$. It follows that $\alpha = \beta= \gamma = 1$ and $C(0, \frac{\pi}{2}) = \sqrt{3}$; therefore, \eqref{eq:error_dominant} simplifies to 
\begin{align}
    \label{eq:error_cartesian}
     &\epsilon_{\M \Xi_{\mathrm{Cart}}, \text{asym}}(f) 
     \notag \\&= \frac{T^2}{12\sqrt{5}} \Big(\int_{\R^2} \big( \frac{\partial^2 f(\V x)}{\partial x_1^2} +\frac{\partial^2 f(\V x)}{\partial x_1 \partial x_2} + \frac{\partial^2 f(\V x)}{\partial x_2^2}\big)^2 \mathrm{d} \V x\Big)^{\frac{1}{2}}.
\end{align}
For the hexagonal grid, it holds that $\theta_1 = 0, \theta_2 = \frac{2\pi}{3}$. It follows that $\alpha = \beta= \frac{\sqrt{3}}{2}$ and $\gamma = 0$, which then yields the elegant formula
\begin{align}
    \epsilon_{\M \Xi_{\mathrm{Hex}}, \text{asym}}(f) 
      &= \frac{T^2}{8\sqrt{15}} \Big(\int_{\R^2} \big(\frac{\partial^2 f(\V x)}{\partial x_1^2} + \frac{\partial^2 f(\V x)}{\partial x_2^2}\big)^2 \mathrm{d} \V x\Big)^{\frac{1}{2}} \notag \\
    &= \frac{T^2}{8\sqrt{15}} \norm{\Delta{f}}_{L_2}. 
\end{align}
where $\Delta{}$ represents the Laplacian operator. We recall that in this case $C(0, \frac{2\pi}{3}) = \sqrt{1.5}$, which is the minimum achieved by $C(\theta_1, \theta_2)$. 
\subsubsection{Error Computation for a Fourier Disk}
In this example, we define the function $f$ through its Fourier transform as
\begin{equation} \label{eq:def_f}
    \hat{f}(\V \omega) = \begin{cases}
      c,  & \sqrt{w_1^2+w_2^2} \leq \omega_{\text{max}} \\
      0,  & \text{otherwise}.
    \end{cases}
\end{equation}
for some $\omega_{\text{max}} > 0$ and $c \in \R$.
We then compute $\epsilon_{\M \Xi, \text{asym}}(f) $ for Cartesian and hexagonal grids with the help of \eqref{eq:error_comp_fourier} as 
\begin{align}
    \epsilon_{\M \Xi_{\mathrm{Cart}}, \text{asym}}(f)  = \frac{T^2 \abs{c} \omega_{\text{max}}^3}{\sqrt{7680\pi}}
\notag \\
    \epsilon_{\M \Xi_{\mathrm{Hex}}, \text{asym}}(f)  = \frac{T^2 \abs{c} \omega_{\text{max}}^3}{\sqrt{11520\pi}}.
\end{align}
Therefore, we have that $\epsilon_{\M \Xi_{\mathrm{Hex}}, \text{asym}}(f)< \epsilon_{\M \Xi_{\mathrm{Cart}}, \text{asym}}(f)$ for a Fourier disk. This observation confirms our claim that hexagonal grids are better in terms of the approximation error.  

\subsection{Box Splines as ReLU Networks \label{part:c}}
Here, we focus on interpolation with box splines on the compact domain $\Omega = (0, 1)^2$. 
Consequently, each function $s \in \mathcal{V}_{T\M I}$ can be constructed using $N = (\frac{1}{T} - 1)^2$ nonzero basis functions for $T < 1$. 

\begin{theorem}
    \label{the:box_spline_as_relu}
    We can represent the Cartesian box spline $\varphi$ using the ReLU$(\cdot) \coloneqq \max(\cdot, 0)$ function as 
    \begin{equation}
        \label{eq:as_relu}
        \varphi(\V x) = \text{ReLU}(1 - \text{ReLU}(x_1-x_2) - \text{ReLU}(x_2) - \text{ReLU}(-x_1))).
    \end{equation}
\end{theorem}
\begin{proof}
    Let us enumerate all configurations of the inequality between $x_1$ and $x_2$ and $0$. Then, we observe that
\eqref{eq:cartesian_spatial} and \eqref{eq:as_relu} are equal in all cases. 
\end{proof}
Theorem \ref{the:box_spline_as_relu} provides the simplest representation for a two-dimensional box spline (a.k.a finite-element basis) through ReLU networks, up to our knowledge \cite{he2022relu, arora2016understanding}. Combining this theorem with $\ref{eq:search_space_with_phi}$, we conclude that any $s \in \mathcal{V}_{T \M I}$ can be constructed using a ReLU network with two hidden layers and $M = 4N$ neurons in total. Moreover, the interpolation error decays at the same rate as \eqref{eq:error_cartesian} where $T = (0.5\sqrt{M}+1)^2$ now depends on the total number of neurons $M$. One can generalize this result to any box spline $B_{\M \Xi}$ by using $B_{\M \Xi}(\V x) = \varphi(\M \Xi^{-1} \V x)$, and proper handling of the domain $\Omega$.

\section{Conclusion}
We have presented an analysis of the approximation error with continuous and piecewise-linear (CPWL) representations using box splines on two-dimensional grids. By deriving explicit error bounds in terms of the grid parameters, we have shown that hexagonal grids minimize the upper bound of the asymptotic error, which emphasizes their optimality in CPWL-based applications.

\bibliographystyle{IEEEtran}
\bibliography{IEEEabrv,dhtv.bib}

\end{document}